\documentclass{article}

\usepackage{verbatim}
\usepackage{amsmath}
\usepackage{amssymb}
\usepackage{amsthm}
\usepackage{comment}
\usepackage[backend=bibtex,style=alphabetic,citestyle=alphabetic]{biblatex}
\usepackage{mathtools}
\usepackage{hyperref}
\usepackage{bbm}
\usepackage[utf8]{inputenc}
\usepackage{tikz-cd}
\usepackage{wasysym}
\usepackage{varwidth}
\usepackage{mathtools}
\usepackage{xspace}
\usepackage{aliascnt}
\usepackage{manfnt}
\usepackage[bbgreekl]{mathbbol}
\usepackage[title,toc]{appendix}

\usepackage{titlesec}
\titleformat{\paragraph}[runin]{\normalfont\normalsize\itshape}{\theparagraph}{}{}[.] 
\titlespacing{\paragraph}{0pt}{0pt}{*1} 

\usepackage[2cell,ps]{xy}
\input xy
\UseAllTwocells
\xyoption{all}

\mathchardef\mhyphen="2D
\DeclareMathOperator{\Hom}{Hom}

\DeclareMathOperator{\pt}{pt}

\DeclareMathOperator{\Cat}{\mathbf{Cat}}

\DeclareMathOperator{\Ob}{Ob}
\DeclareMathOperator{\Exact}{Exact}
\DeclareMathOperator{\Grid}{Grid}

\DeclareMathOperator{\Rep}{Rep}
\newcommand{\Vect}{\mathbf{Vect}}
\newcommand{\VectNu}{\mathbf{Vect}_{\ZZ[\nu,\nu^{-1}]}}

\DeclareMathOperator{\Id}{Id}

\DeclareMathOperator{\Spaces}{Stacks}
\newcommand{\Stacks}{\Spaces}
\DeclareMathOperator{\sset}{sSet}

\DeclareMathOperator{\Corr}{Corr}
\newcommand{\CorrStacks}{\Corr(\Stacks)}

\DeclareMathOperator{\HallAlg}{H_{\CCC}}

\newcommand{\dashmod}{\text{-mod}}
\newcommand{\BigCorr}{\operatorname{CORR}}

\newcommand{\forceNewLine}{\leavevmode\newline}
\newcommand{\StackFunc}{\mathfrak{F}}

\DeclareMathOperator{\Kgroup}{K}

\newcommand{\Ainf}{A_{\infty}}

\DeclareMathOperator{\Dmod}{D-mod}

\DeclareMathOperator{\Ext}{Ext}

\DeclareMathOperator{\Spc}{Spc}
\newcommand{\Maps}{\mathcal{M}aps}

\newcommand{\CatTransfer}{T^{\Cat}}

\newcommand{\VectNuTransfer}{T^{\VectNu}}

\newcommand{\DKTransfer}{T^{DK}}
\newcommand{\KWeilTransfer}{T^W}

\DeclareMathOperator{\Perf}{Perf}

\newcommand{\AinfEnd}{\mathcal{E}nd}

\newcommand{\QuantumGroup}{U_{q}(\mathfrak{g})}
\newcommand{\Zquantum}{\dot{U}^{-}_{\nu}(\mathfrak{g}_{\Quiver})}

\newcommand{\CCC}{\mathcal{C}}
\newcommand{\PPP}{\mathcal{P}}
\newcommand{\SSS}{\mathcal{S}}
\newcommand{\DShladic}[1]{\mathcal{D}({#1}_{\kk})}
\newcommand{\DShcomplex}[1]{\mathcal{D}({#1}_{\CC})}
\newcommand{\QQQ}{\mathcal{Q}}
\newcommand{\Mod}{Mod}

\newcommand{\LusztigAlgebroid}{\mathfrak{L}}
\newcommand{\Lusztig}{\mathcal{L}}
\newcommand{\HallCat}{\mathcal{H}}
\newcommand{\Quiver}{\overrightarrow{Q}}

\newcommand{\geoHall}{H_{geo}}
\newcommand{\combHall}{H_{comb}}

\newcommand{\AAA}{\mathcal{A}}

\newcommand{\aug}[1]{A(#1)}
\newcommand{\grid}{\operatorname{grid}}

\newcommand{\OrdSet}{\mathbb{\Delta}}

\newcommand{\AugOrdSet}{\mathbb{\Delta}_+}

\DeclareMathOperator{\Tr}{Tr}
\DeclareMathOperator{\Frob}{Fr}
\newcommand{\FFF}{\mathcal{F}}

\newcommand{\OOO}{\mathcal{O}}
\newcommand{\FF}{\mathbbm{F}}

\newcommand{\NN}{\mathbbm{N}}
\newcommand{\ZZ}{\mathbbm{Z}}
\DeclareMathOperator{\point}{\pt}
\newcommand{\kk}{\mathbbm{k}}
\newcommand{\CC}{\mathbbm{C}}

\newcommand{\Alg}[1]{A_{#1}}

\newcommand{\ord}[1]{\left[ #1 \right] }
\newcommand{\ordop}[1]{\Delta_{#1} }

\newcommand{\field}{\mathbbm{k}}

\DeclareMathOperator{\Aut}{Aut}
\newcommand{\tik}{\begin{tikzcd}}
\newcommand{\tak}{\end{tikzcd}}
\newcommand{\Nerve}[1]{N(#1)}

\newcommand{\gridGR}[1]{\operatorname{CGrid}_{#1}}

\makeatletter
\def\latearrow#1#2#3#4{%
  \toks@\expandafter{\tikzcd@savedpaths\path[/tikz/commutative diagrams/every arrow,#1]}%
  \global\edef\tikzcd@savedpaths{%
    \the\toks@%
    (\tikzmatrixname-#2)
    to%
    node[/tikz/commutative diagrams/every label] {$#4$}
    (\tikzmatrixname-#3)
;}}
\makeatother

\def\stik#1#2{
\begin{tikzpicture}[baseline= (a).base]%
\node[scale=#1] (a) at (0,0){%
\begin{tikzcd}[ampersand replacement=\&]%
#2
\end{tikzcd}%
};%
\end{tikzpicture}
}

\def\cnstik#1#2#3{
\begin{tikzpicture}[baseline= (a).base]%
\node[scale=#1] (a) at (0,0){%
\begin{tikzcd}[ampersand replacement=\&,column sep=#2em]%
#3
\end{tikzcd}%
};%
\end{tikzpicture}
}

\newtheorem{Theorem}{Theorem}

\newtheorem*{unnu-theorem}{Theorem}

\newtheorem{Proposition}{Proposition}[section]

\newaliascnt{Corollary}{Proposition}
\newtheorem{Corollary}[Corollary]{Corollary}
\aliascntresetthe{Corollary}

\newaliascnt{Lemma}{Proposition}
\newtheorem{Lemma}[Lemma]{Lemma}
\aliascntresetthe{Lemma}

\newtheorem* {Claim}{Claim}

\theoremstyle{definition}
\newaliascnt{Definition}{Proposition}
\newtheorem{Definition}[Definition]{Definition}
\aliascntresetthe{Definition}

\theoremstyle{remark}
\newaliascnt{Notation}{Proposition}
\newtheorem{Notation}[Notation]{Notation}
\aliascntresetthe{Notation}

\newaliascnt{Remark}{Proposition}
\newtheorem{Remark}[Remark]{Remark}
\aliascntresetthe{Remark}

\newaliascnt{Example}{Proposition}
\newtheorem{Example}[Example]{Example}
\aliascntresetthe{Example}

\def\sectionautorefname~#1\null{%
\S#1\null
}

\def\equationautorefname~#1\null{%
(#1)\null
}

\begin{document}
\numberwithin{equation}{section}

\title{Hall categories and KLR categorification}
\author{Adam Gal, Elena Gal, Kobi Kremnizer}
\maketitle

\begin{abstract}
This paper is the first step in the project of categorifying the bialgebra structure on the half of quantum group $\QuantumGroup$ by using geometry and Hall algebras. We equip the category of D-modules on the moduli stack of objects of the category $\Rep_{\CC}(\Quiver)$ of representations of a quiver with the structure of an algebra object in the category of stable $\infty$-categories. The data for this construction is provided by an extension of the Waldhausen construction for the category $\Rep_{\CC}(\Quiver)$. We discuss the connection to the  Khovanov-Lauda-Rouquier categorification of  half of the quantum group $\QuantumGroup$ associated to the quiver $\Quiver$ and outline our approach to the categorification of the bialgebra structure.     
\end{abstract}

\tableofcontents

\section{Introduction}

This work is part of a program to categorify quantum groups and their representations using geometry.

The link between quantum groups and geometry was made in the works of Ringel and Lusztig. In \cite{ringel1990hall} Ringel famously related the half of the quantum group $U_{\sqrt{q}}(\mathfrak{g}_{\Quiver})$ associated to a quiver $\Quiver$ to the \emph{Hall algebra} of the category $\Rep_{\FF_q}(\Quiver)$. 
Slightly later Lusztig's work involving perverse sheaves on stacks of objects of the category $\Rep_{\FF_q}(\Quiver)$ lead to the discovery of a canonical basis for the quantum group. 
Related geometric considerations were used by Khovanov, Lauda and Rouquier in \cite{KL1,KL2,R} to construct categorified representations of quantum groups. 

In the present work we extend Lusztig's construction in  (\cite{LusztigCanonicalBases},\cite{LusztigQuivers},$\ldots$) to the $\infty$-categorical setting.
It is known from Lusztig's work that certain moduli spaces of flags in $\Rep_{\FF_q}(\Quiver)$ can be used to construct an additive monoidal category which categorifies $U^-_{\sqrt{q}}(\mathfrak{g}_{\Quiver})$. The underlying idea of our present work is that the Waldhausen construction for the category of representation of the quiver $\Quiver$ precisely encodes the information about the monoidal structure on a certain stable $\infty$-category which we call $\HallCat$.
The category $\HallCat$ is therefore a natural object to consider if one wants to categorify the bialgebra structure on $U^-_{\sqrt{q}}(\mathfrak{g}_{\Quiver})$. Our construction of this monoidal structure is given in a language that allows for a natural extension that incorporates the categorical version of the bialgebra condition as we elaborate below in \autoref{Int:bialgebra}.
The bialgebra structure is important because it is instrumental for the construction of the tensor product of categorified representations of quantum groups.

We construct a stable monoidal $\infty$-category $\HallCat$, whose Grothendieck K-group satisfies 
\[K(\HallCat)\cong \dot{U}^-_{\sqrt{q}}(\mathfrak{g}_{\Quiver})\]
The above isomorphism is an isomorphism of algebras. The monoidal structure on $\HallCat$, i.e. the product and the data of the higher coherences is given by the Waldhausen construction for the category $\Rep_{\CC}{\Quiver}$ of complex representations of the quiver $\Quiver$. We use the notion of a \emph{2-Segal space} introduced by Dyckerhoff and Kapranov in \cite{KapranovDyckerhoff} to show that the Waldhausen construction produces an $(\infty,1)$- category of correspondences of stacks which encodes the above monoidal structure in $\autoref{geoHall}$. This system of correspondences equips the stable $\infty$-category of D-modules on the stack of objects of $\Rep_{\CCC}(\Quiver)$ with a product satisfying natural system of coherences as shown in \autoref{Transfer}.

We explain the relation of $\HallCat$ to the Khovanov-Lauda-Rouquier (KLR) categorification of quantum groups in \autoref{secKLR}. To summarize, $\HallCat$ is equivalent to the  stable $\infty$-category of perfect modules over a KLR algebra rather than the category of graded modules as in the KLR approach. 

The natural next step is the the construction of the categorified bialgebra structure on $\HallCat$. We expect that the above approach would have an advantage in this respect since, as we explain in \autoref{subKLR}, passing to the KLR construction requires formality of certain bimodules $M_{\nu,\nu'}$ defining the product on $\HallCat$ and it is unclear whether the counterparts of $M_{\nu,\nu'}$ for the co-multiplication are formal and also whether the bimodules expressing the bialgebra condition are formal in this picture (formality in this situation is deduced from purity, which is only assured for the multiplication. cf. \cite[\S3.7]{SchiffmannHallCategory}). If these modules indeed aren't formal there is little hope to construct the bialgebra structure in the KLR language (\autoref{remFormality}).

In more detail, the algebra structure on $\HallCat$ is given by a monoidal functor of $(\infty,1)$-categories 
\[N(\AugOrdSet) \rightarrow \Cat, \ord{1}\mapsto\HallCat\]
where $N(\AugOrdSet)$ denotes the nerve of the category of ordered sets with product given by union of sets and $\Cat$ is the category of stable $\infty$-categories. The construction of this functor is given in essentially combinatorial terms (see $\autoref{Hcomb}$). The construction of the bialgebra structure involves extending this functor to the monoidal category generated by two copies of $\AugOrdSet$.  This category can be thought of as the category whose objects are posets or more precisely "grids". Hence we need to extend our construction from linearly ordered sets to grids.
It is easy to generalize the formulas used in $\autoref{Hcomb}$ so that they  would apply to this class of posets. However to prove that this extension produces a monoidal functor one needs to introduce some additional considerations. This is the subject of our article  \cite{ourGeometricHall2}.

\subsection*{Background}
Lusztig's construction of the canonical basis for $U_{\sqrt{q}}(\mathfrak{g}_{\Quiver})$ essentially involved constructing an additive category $\QQQ$, equipped with product and co-product categorifying multiplication and co-multiplication of $U_{\sqrt{q}}(\mathfrak{g}_{\Quiver})$ (see for example \cite{SchiffmannHallCategory} for a concise presentation of this point of view). The word "categorifying" here precisely means that 
\[K(\QQQ)\cong \dot{U}^-_{\sqrt{q}}(\mathfrak{g}_{\Quiver})\]
where the above is an isomorphism of algebras and co-algebras. The category $\QQQ$ is usually called \emph{Hall category} in the literature. An article of Varagnolo and Vasserot \cite{VaragnoloVasserot} relates $\QQQ$ to the work of Khovanov, Lauda and Rouquier on the categorification of quantum groups and their representations. A missing piece of the puzzle (as far as the autors are aware) is the relation between the multiplication and co-multiplication on the categorical level.

It is known that the multiplication and co-multiplication in the Hall algebra associated to  $\Rep_{\FF_q}(\Quiver)$ are related by Green's theorem\cite{Green}, making it into a bialgebra "up to a twist". 
In \autoref{Int:green twist} we make the observation that the geometric origin of this twist lies in the relation between the upper star and upper shriek pull-back functors between sheaves on certain stacks associated to $\Rep_{\FF_q}(\Quiver)$. 
Expressing this bialgebra condition on the categorical level is a step on the way to the categorification of the tensor product of categorified representations of quantum groups constructed in \cite{KL1,KL2,R}.
In other words, one would like to construct a "higher bialgebra" categorifying the bialgebra $U^-_q(\mathfrak{g}_{\Quiver})$.  

To achieve this objective one needs to use more sophisticated machinery than the one involved in the construction of $\QQQ$.
For technical reasons having to do with using adjoint functors and base change it is natural to work in the world of dg- or stable $\infty$-categories rather than with derived categories (of sheaves on our geometric objects). In essence one needs to consider a smallest stable $\infty$-category containing $\QQQ$.

Working with $\infty$-categories comes at the cost of needing to specify higher coherences. For example to equip such a category with a product one needs to provide all higher coherences in place of providing the associator and checking the pentagon axiom. In the categorification of the quantum group the product is constructed by considering it as a Hall algebra associated to the category $\Rep(\Quiver)$. Luckily, the construction of the categorified Hall algebra product for this category has all these higher coherences naturally "built in". Essentially they are provided by the Waldhausen construction for this category. This natural higher associativity structure for various examples of Hall algebras is studied in the work \cite{KapranovDyckerhoff} by Dyckerhoff and Kapranov. 

The main deficit we aim to cover in this article is the current lack of a coherent (and straightforward) construction of this system of higher coherences, the goal being to extend this system to a higher bialgebra structure. This is expanded on more in \autoref{Int:bialgebra}.

\subsection{Higher Associativity and 2-Segal conditions}
The insight of \cite{KapranovDyckerhoff} was to split the associative algebra construction for the Hall algebra into geometric and algebraic parts. In the geometric part associativity is expressed by\emph{ 2-Segal conditions} introduced in \cite{KapranovDyckerhoff} and, independently, under the name of \emph{decomposition spaces} in the series of articles \cite{decomposition1,decomposition2,decomposition3}. 

Denote the stacks of objects and exact sequences (or length 1 flags) in $\Rep(\Quiver)$ as $\Ob$ and $\Exact$ respectively. Consider the correspondence
\[
\stik{1}{
{} \& \Exact \ar{dl}[above, xshift=-0.5em]{end} \ar{dr}{mid} \& {} \\
\Ob\times \Ob \& {} \& \Ob
}
\]
This correspondence is a geometric object which encodes multiplication in the Hall category. Namely, by taking the composition of appropriate pull-and push- functors we obtain the product $\AAA\otimes\AAA\rightarrow\AAA$, where $\AAA$ is the category of constructible $l$-adic sheaves when the representations are taken over $\FF_q$, as in the work of Lusztig, or D-modules when the representations are taken over $\CC$ as in this article.

Moreover, the associativity of this product is given by considering the stack of flags of length 2.
The fact that it defines an associator morphism is a consequence of the Waldhausen simplicial stack for $\Rep(\Quiver)$ being \emph{2-Segal}, and base change for our choice of push- and pull-functors. Higher associators are represented by longer flags.

In the present work we assemble this data into a functor 
\[N(\AugOrdSet) \rightarrow \Corr(\Stacks)\] in \autoref{geoHall}. Here $\Corr(\Stacks)$ is an $(\infty,1)$-category of correspondences with morphisms being grids of pull-back squares. As such $\Corr(\Stacks)$ is well adapted to turning composition of correspondences into an associative composition of regular morphisms. $\Corr(\Stacks)$ appears in the work \cite{gaitsbook} by Gaitsgory and Rozenblyum, and we use their results to define a monoidal functor from $\Corr(\Stacks)$ to stable $\infty$-categories, which assigns to a stack its category of D-modules. Our main result in \autoref{geoHall} is \autoref{2Segalequivalent} which shows that our construction produces a monoidal functor of $(\infty,1)$-categories iff the simplicial Waldhausen stack is 2-Segal. 

As we mention above, to extend this to a bialgebra construction we need more involved combinatorics. 
To this end we consider the construction called the \emph{abstract Hall algebra} in \cite{KapranovDyckerhoff} and split it into a \emph{combinatorial} and a \emph{geometric} part.
The combinatorial Hall algebra $\combHall$ is introduced in \autoref{Hcomb}. 
Working with $\combHall$ makes the construction of the algebra structure more transparent, but more importantly it is indispensable for our construction of the bialgebra structure in \cite{ourGeometricHall2}. 

\subsection{The bialgebra structure}
\label{Int:bialgebra}
We would like to advertise here the next paper in our project \cite{ourGeometricHall2}. 
Consider the categorified bialgebra structure given by extending the functor defining the algebra structure on $\HallCat$ to a functor from the category $\AugOrdSet\otimes\AugOrdSet$ of grid posets.

The first nice thing is that in this way we naturally recover constructions that appear in the literature, and place them in an inherently coherent framework.
For example: The stack $\Grid$ of $3\times 3$ grids of short exact sequences appears in the proof of \cite[Theorem 1]{Green} about the twisted bialgebra structure of the Hall algebra and also naturally appears in our construction as part of the data of a bialgebra object.
Specifically it appears as the image of the $2\times 2$ grid poset.

The next upshot is that this point of view naturally provides us with a candidate for a 2-morphism which replaces the bialgebra relation on the categorical level.
It is also easy to show that this morphism is an isomorphism. Let us explain this in more detail:

The categorification of the bialgebra condition can be expressed using the notion of Beck-Chevalley square. In \cite{ourSSH} we proposed a notion of higher bialgebra which in rough terms is given by the following: 

Consider a category $\CCC$ supplied with two multiplicative (i.e. monoidal) structures $m_1,m_2$ and a square 
\begin{equation}
    \label{Eqn:GreenSquare}
\stik{1}{
\CCC^{\otimes 4} \ar{r}{m_1^{\otimes 2}} \ar{d}[left]{m_2^{\overline{\otimes 2}}} \& \CCC^{\otimes 2} \ar{d}{m_2} \arrow[shorten >=0.4cm,shorten <=0.4cm,Rightarrow]{dl}[above,sloped]{\alpha}\\
\CCC^{\otimes 2} \ar{r}{m_1} \& \CCC}
\end{equation}
with the \emph{condition} that this square is Beck-Chevalley, i.e. that adjoints of the functors exist, and that replacing verticals (or equivalently horizontals) with adjoints yields an isomorphism.

It seems that in various cases the construction of the square involving multiplications is more natural than one involving a multiplication and comultiplication and in this way it also is easier to construct a coherent system.

In the case of the Hall algebra both multiplications are the same and the condition we need to check comes down to checking the Beck-Chevalley condition for D-modules on the stacks forming the square:
\begin{equation}
\label{Int:GridSquare}
\stik{1}{
\Grid \ar{r}{q_1} \ar{d}{q_2} \& (\Exact)^2 \ar{d}{p_2} \\
(\Exact)^2 \ar{r}{p_1} \& (\Ob)^4
}
\end{equation}
where the maps from $\Grid$ are the projections to the top and bottom or left and right short exact sequences.

This square has the property that the canonical map $i$ from $\Grid$ to the pullback of the lower right corner has $i^*$ fully faithful (an explicit computation using this map appears in \cite[\S 2.4]{Dyckerhoff}). This implies that the square of categories

\[
\stik{1}{
\Dmod(\Grid) \ar[leftarrow]{r}{q_1^*} \ar[leftarrow]{d}{q_2^*} \& \Dmod((\Exact)^2) \ar[leftarrow]{d}{p_2^*} \arrow[shorten >=0.4cm,shorten <=0.4cm,equals]{dl}[above,sloped]{\sim}\\
\Dmod((\Exact)^2) \ar[leftarrow]{r}{p_1^*} \& \Dmod((\Ob)^4)
}
\]
is Beck-Chevalley since the associated square
\[
\stik{1}{
\Dmod(\Grid) \ar{r}{{q_1}_*} \ar[leftarrow]{d}{q_2^*} \arrow[shorten >=0.4cm,shorten <=0.4cm,equals]{dr}[above,sloped]{\sim} \& \Dmod((\Exact)^2) \ar[leftarrow]{d}{p_2^*} \\
\Dmod((\Exact)^2) \ar{r}{{p_1}_*} \& \Dmod((\Ob)^4)
}
\]
commutes.

\subsection{The twist from Green's Theorem}

\label{Int:green twist}
In \cite[Theorem 1]{Green} there appears a twist of the multiplication by a power of a constant (the square root of the size of the field) which is required to make the Hall algebra into a bialgebra. To give an example of the power of the geometric approach let us explain how to recover this twist as a straightforward consequence of the higher bialgebra condition together with an essentially arithmetical computation.

Note that the comultiplication of the Hall algebra involves not lower star but lower shriek functors. In the case when $\CCC$ is hereditary, the maps in the square \autoref{Int:GridSquare} are all smooth (see \autoref{HereditarySm-PrProposition}), and so upper star has a left adjoint which is the shift of lower shriek by twice the relative dimension of the map. Since the right Beck-Chevalley condition implies the left one, we get an isomorphism 
\[
\stik{1}{
\Dmod(\Grid) \ar[leftarrow]{r}{q_1^*} \ar{d}[left]{{q_2}_![2d_{q_2}]} \arrow[shorten >=0.4cm,shorten <=0.4cm,equals]{dr}[above,sloped]{\sim} \& \Dmod((\Exact)^2) \ar{d}{{p_2}_![2d_{p_2}]} \\
\Dmod((\Exact)^2) \ar[leftarrow]{r}{p_1^*} \& \Dmod((\Ob)^4)
}
\]

A straightforward computation shows that $2d_{q_2}-2d_{p_2}$ is exactly the power involved in Green's twist.

\begin{Remark}
In the above argument we used the hereditary property of $\CCC$, just as is required in Green's theorem. This was necessary in order to deduce a relation on the level of vector spaces, as the lower star functor doesn't in general translate to anything reasonable on that level, but the lower shriek does. However, it is reasonable to expect that in more general situations we will be able to produce interesting relations using this same right mate-left mate argument.
\end{Remark}

\subsection{Further Remarks}
In this article we are concerned with the construction of Waldhausen stacks and Hall categories for $\Rep_{\CC}(\Quiver)$. However the fact that the Waldhausen simplicial stack is 2-Segal is true for a larger class of categories, see \cite{KapranovDyckerhoff}. Thus our construction of a "higher Hall algebra" structure works in greater generality than the category of representations of a quiver. For example one can apply this construction to obtain a stable $\infty$- Hall category associated to to the category of coherent sheaves on a smooth projective curve $X$ (see \cite{schiffmann2006canonical,SchiffmannHallCategory} for more details on Hall category associated to this object). 

Since the introduction of the 2-Segal/decomposition spaces there appeared numerous works developing this circle of ideas which we cannot adequately survey here. Our article \cite{ourCubes} and also \cite{Penney2Segal} carry out a similar construction to \autoref{Hcomb} from a more abstract point of view.
\section{Notations}

\begin{itemize}
\item $\OrdSet$ - the category of finite ordered sets. The elements of $\OrdSet$ will be denoted by \[\ordop{0}=0, \ordop{1}=0\rightarrow 1, \ordop{2}=0\rightarrow 1\rightarrow 2, \ldots\]
\item $\AugOrdSet$ - the augmented category of finite ordered sets. The elements of $\AugOrdSet$ will be denoted by \[\ord{0}=\emptyset, \ord{1}=\{0\}, \ord{2}=\{0\rightarrow 1\}, \ord{3}=\{0\rightarrow 1\rightarrow 2\}, \ldots\]
\item $\Quiver=(I,H,s,t)$ - a quiver with set of vertices $I$, set of edges $H$ and source and target maps $s$ and $t$. We don't allow loop edges in $\Quiver$.
\item $\kk$ - the algebraically closed field $\overline{\FF_q}$
\item For a theory with transfer $T$ we denote the resulting algebra object by $\Alg{T}$ (see \autoref{def:AlgebraObject}).
\end{itemize}
\section{Background}

\subsection{Waldhausen stacks}
\label{WaldhausenStacksSection}
The classical Waldhausen construction \cite{Waldhausen} associates a simplicial groupoid to an exact category. This construction is of interest to us since, as noted in \cite{KapranovDyckerhoff}, it can be used to define multiplication in the Hall algebra and provide the associativity data for it. We describe a version of this construction producing a system of stacks associated to $\CCC:=\Rep_{\CC}{\Quiver}$ - the category of representations of the quiver $\Quiver$.

Our construction will use the notion from \cite[\S 7.1]{JoyceConfigurations} of a functor of \emph{families of objects} in $\CCC$. Consider the functor $U\mapsto \FF_{\CCC}(U)$ where $\FF_{\CCC}(U)$ is the exact category of all locally free sheaves on $U$ equipped with an action of the path algebra of the quiver. It is shown in \cite{JoyceConfigurations} that $\FF_{\CCC}$ is a functor of families of objects. In particular it means that $\FF_{\CCC}$ is a stack of exact categories, satisfying some flatness conditions outlined in \cite{JoyceConfigurations}.
This example generalizes to any finitely generated associative algebra.

We define a system of stacks $S_{\CCC}:\OrdSet^{op}\to \Spaces$ as follows:

\begin{enumerate}
    \item Take $X\in\OrdSet$
    \item Consider $\grid(X):=\Hom(0\to 1,X)$, as a marked category, with the constant maps the marked objects. Note that $\grid(X)$ has two classes of maps, the "horizontal" and the "vertical", by which we mean maps that are identity on the 0 or 1 component, respectively. 
    \item Define $S_X\CCC$ to be the stack of maps from $\grid(X)$ to $\CCC$ which take the marked objects to $0$, the horizontal maps to monomorphisms and the vertical maps to epimorphisms, and take Cartesian squares to Cartesian squares. More precisely, the stack $S_X\CCC$ is given by the assignment sending $U$ to the space of maps as above from $\grid(X)$ to $\FF_\CCC(U)$.

\end{enumerate}
It is shown in \cite{JoyceConfigurations} that the family of moduli stacks defined above in 3. are \emph{Artin} stacks locally of finite type. In this article we will consider them as objects in the $\infty$-category of derived stacks (see \cite{hag2}) which we denote $\Stacks$.  

\begin{Example}
Take $X=\ordop{1}=0\to 1$, then \[
\grid(X)=\stik{1}{
00 \ar{r} \& 01 \ar{d} \\
{} \& 11
}
\]
and so $S_X\CCC=S_{\ordop{1}}\CCC$ is the stack of objects of $\CCC$. In other words, $S_{\ordop{1}}\CCC$ is the moduli space of representations of the quiver $\Quiver$, i.e. a disjoint union of quotient stacks. 
\end{Example}

\begin{Example}
Take $X=\ordop{2}=0\to 1\to 2$, then  \[
\grid(X)=\stik{1}{
00 \ar{r} \& 01 \ar{r} \ar{d} \& 02 \ar{d} \\
{} \& 11 \ar{r} \& 12 \ar{d} \\
{} \& {} \& 22
}
\]
The data of a map $\grid(X)\to\CCC$ then consists of a square\[
\stik{1}{
C_{01} \ar[hook]{r} \ar[two heads]{d} \& C_{02} \ar[two heads]{d}\\
C_{11}=0 \ar[hook]{r} \& C_{12}
}
\]
which must be Cartesian and therefore also coCartesian. This just means that $C_{01}\to C_{02} \to C_{12}$ is an exact sequence.
In all, $S_X\CCC=S_{\ordop{2}}\CCC$ is the stack of exact sequences in $\CCC$, or in other words moduli space of flags of length 1.
\end{Example}
It is now easy to guess the general shape of $S_X\CCC$. Namely, for $X=\ordop{n}$ we get the diagrams of the form 
\[
\stik{1}{
0 \ar[hook]{r} \& C_{01} \ar[two heads]{d} \ar[hook]{r} \& C_{02}\ar[two heads]{d} \ar[hook]{r} \& \cdots  \ar[hook]{r} \& C_{0n}\ar[two heads]{d} \\
{} \& 0  \ar[hook]{r} \& C_{12}\ar[two heads]{d} \ar[hook]{r} \& \cdots  \ar[hook]{r} \& C_{1n}\ar[two heads]{d} \\
{} \& {} \& 0 \ar[hook]{r} \& \cdots  \ar[hook]{r} \& C_{2n}\ar[two heads]{d}\\
{} \& {} \& {} \& \ddots \& \vdots\ar[two heads]{d} \\
{} \& {} \& {} \& {} \& 0
}
\]
where every square is biCartesian.

It is shown in \cite{KapranovDyckerhoff} [Lemma 2.4.9] that the groupoid of diagrams of this shape is equivalent to the groupoid of flags of length $n-1$ providing the connection to the classical Waldhausen construction. The argument generalizes to stacks in a straightforward manner. Hence $S_{\ordop{n}}\CCC$ is isomorphic to a moduli stack of flags of length $n$ in $\CCC$. We note that these stacks are disjoint unions of quotient stacks.

In the sequel we will use the  following extended version of Waldhausen construction:

\begin{Lemma}
\label{S_ext}
The right Kan extension of $S$ along the Yoneda embedding functor $\OrdSet^{op} \rightarrow \sset^{op}$ exists.
\end{Lemma}
\begin{proof}
The category $\Spaces$ is complete.
\end{proof}
\begin{Notation}
From now on we will shorten $S_{\ordop{n}}\CCC$ to $S_{n}$.
\end{Notation}
\subsection{The 2-Segal conditions}
\label{2Segal}
2-Segal conditions were introduced in \cite[\S 2.3]{KapranovDyckerhoff}. Let us formulate this notion for a simplicial stack $S$.

Let $S$ be a functor $S:\OrdSet^{op}\to\Stacks$, and $P$ a polygonal decomposition of an $n$-gon, written as $(P_1,\ldots,P_k)$. e.g.
\[
\begingroup%
  \makeatletter%
  \providecommand\color[2][]{%
    \errmessage{(Inkscape) Color is used for the text in Inkscape, but the package 'color.sty' is not loaded}%
    \renewcommand\color[2][]{}%
  }%
  \providecommand\transparent[1]{%
    \errmessage{(Inkscape) Transparency is used (non-zero) for the text in Inkscape, but the package 'transparent.sty' is not loaded}%
    \renewcommand\transparent[1]{}%
  }%
  \providecommand\rotatebox[2]{#2}%
  \ifx\svgwidth\undefined%
    \setlength{\unitlength}{156.92625631bp}%
    \ifx\svgscale\undefined%
      \relax%
    \else%
      \setlength{\unitlength}{\unitlength * \real{\svgscale}}%
    \fi%
  \else%
    \setlength{\unitlength}{\svgwidth}%
  \fi%
  \global\let\svgwidth\undefined%
  \global\let\svgscale\undefined%
  \makeatother%
  \begin{picture}(1,0.75123049)%
    \put(0,0){\includegraphics[width=\unitlength,page=1]{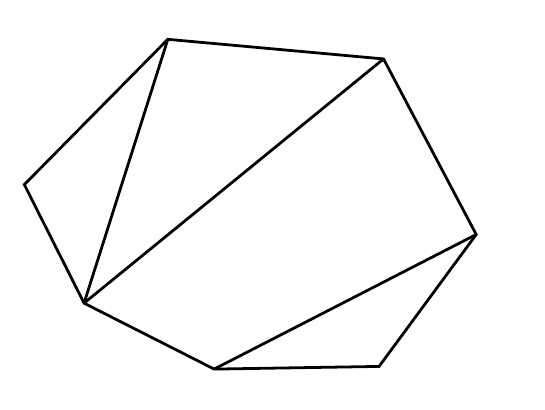}}%
    \put(0.12172059,0.40007091){\color[rgb]{0,0,0}\makebox(0,0)[lb]{\smash{}}}%
    \put(0.08573423,0.39352797){\color[rgb]{0,0,0}\makebox(0,0)[lb]{\smash{$P_1$}}}%
    \put(0.32618869,0.52438755){\color[rgb]{0,0,0}\makebox(0,0)[lb]{\smash{$P_2$}}}%
    \put(0.44069086,0.29701911){\color[rgb]{0,0,0}\makebox(0,0)[lb]{\smash{$P_3$}}}%
    \put(0.5748219,0.1203587){\color[rgb]{0,0,0}\makebox(0,0)[lb]{\smash{$P_4$}}}%
    \put(-0.0042317,0.41315677){\color[rgb]{0,0,0}\makebox(0,0)[lb]{\smash{$0$}}}%
    \put(0.26893765,0.71249808){\color[rgb]{0,0,0}\makebox(0,0)[lb]{\smash{$1$}}}%
    \put(0.7073172,0.67651165){\color[rgb]{0,0,0}\makebox(0,0)[lb]{\smash{$2$}}}%
    \put(0.90687802,0.32482657){\color[rgb]{0,0,0}\makebox(0,0)[lb]{\smash{$3$}}}%
    \put(0.67460231,0.02712161){\color[rgb]{0,0,0}\makebox(0,0)[lb]{\smash{$4$}}}%
    \put(0.3327317,0.0074925){\color[rgb]{0,0,0}\makebox(0,0)[lb]{\smash{$5$}}}%
    \put(0.06937679,0.14489469){\color[rgb]{0,0,0}\makebox(0,0)[lb]{\smash{$6$}}}%
  \end{picture}%
\endgroup%

\]

We define $S_P$ to be $S_{P_1}\times_{S_{P_1\cap P_2}} S_{P_2}\times_{S_{P_2\cap P_3}} S_{P_3}\ldots \times_{S_{P_{k-1}\cap P_k}} S_{P_k}$ where $S_{P_i}$ is $S_{\ordop{\#\{\text{vertices of }P_i\}}}$ and note that by functoriality we have a natural map
$S_{\ordop{n}}\xrightarrow{\alpha_P} S_P$.

\begin{Definition}
\label{2SegalAdaptednessDef}
We say that $S$ is \emph{adapted} to a polygonal decomposition $P$ if the map $\alpha_{P}$ is a weak equivalence
\end{Definition}

\begin{Definition}
A functor $S:\OrdSet^{op}\to\Stacks$ is said to be a 2-Segal stack if $S$ is adapted to any polygonal decomposition $P$ of an $n$-gon for any $n\geq 3$.
\end{Definition}

In \cite[2.4.8]{KapranovDyckerhoff} the authors prove that the simplicial groupoid given by the Waldhausen construction for an exact category is 2-Segal. Their proof clearly generalizes to simplicial stacks defined in \autoref{WaldhausenStacksSection}.

\subsection{The category of correspondences}
\label{CorrespondencesSection}
We would like to define several kinds of categories of correspondences that we will use in our constructions.

Let $\CCC$ be an $(\infty,1)$-category, and $\SSS,\PPP$ two classes of maps closed on composition in $\CCC$. Let us briefly recall the construction of the $(\infty,1)$ category of correspondences in $\CCC$ from \cite[Chapter V.1]{gaitsbook}.

$\Corr{\CCC}_{\SSS,\PPP}$ is a \emph{complete Segal space} i.e. an object in $\Spc^{\OrdSet^{op}}$ defined as follows:

Consider the category $\grid(\ordop{n})$ of \autoref{WaldhausenStacksSection}. Reversing the horizontal arrows yields a category $\gridGR{n}$. We define \[
\left(\Corr{\CCC}_{\SSS,\PPP}\right)_{n}=\Maps_{\SSS,\PPP}(\gridGR{n},\CCC)
\]
where the notation $\Maps_{\SSS,\PPP}$ means the space of maps sending horizontal maps to $\SSS$, vertical maps to $\PPP$, and squares to pullback squares. i.e. we have diagrams of the form 
\[
\stik{1}{
\bullet  \& \bullet \ar{l}{s} \ar{d}{p} \& \bullet \ar{d}{p} \ar{l}{s} \ar[no head,dotted]{r}\& \bullet \ar{d}{p} \& \ar{l}{s}\bullet \ar{d}{p}\\
{}  \& \bullet \& \bullet  \ar[no head,dotted]{r} \ar[no head,dotted]{dr}\ar{l}{s} \& \bullet \ar[no head,dotted]{d}  \& \bullet \ar{l}{s} \ar[no head,dotted]{d}\\
{} \& {} \& {} \& \bullet \& \bullet \ar{l}{s} \ar{d}{p}\\
{} \& {} \& {} \& {} \& \bullet
}
\]
The face and degeneracy maps are defined by restrictions and insertions.
\begin{Definition}\forceNewLine
\label{def:corr}
Denote by $\Corr(\Stacks)$ the $(\infty,1)$-category $\Corr(\Stacks)_{all,all}$.
\end{Definition}
In this article we will also consider $\SSS$ and $\PPP$ smooth and proper morphisms of stacks respectively in \autoref{Transfer}. The idea is that $\SSS$ maps are well adapted to pullback, $\PPP$ maps are well adapted to pushforward. Altogether this gives the data required for turning a composition of correspondences into a composition of regular morphisms.

For purposes of the construction in \autoref{geoHall} we also introduce the following:

\begin{Definition}
Let $\CCC$ be an an $(\infty,1)$-category (resp. an ordinary category). Denote by $\BigCorr(\CCC)$ the simplicial space (the simplicial set) defined in the same way as $\Corr{\CCC}_{all,all}$ but without the pullback requirement.
\end{Definition}

\section{The geometric construction for Hall categories}
\label{geoHall}

\subsection{The geometric approach to multiplication in Hall categories}

The information about multiplication and co-multiplication in Hall algebra or Hall category is encoded geometrically by the correspondence 

\[
\stik{1}{
{} \& S_2 \ar{dl}[above, xshift=-0.5em]{end} \ar{dr}{mid} \& {} \\
S_1\times S_1 \& {} \& S_1
}
\]
where the stacks $S_1$ and $S_2$ are the stacks of objects of $\CCC$ and  short exact sequences of $\CCC$ respectively, given by the Waldhausen construction and the maps $mid:S_2 \to S_1$ and $end:S_2 \to S_1\times S_1$ are given by 
\begin{align*}
    mid(0\to U\to V\to W\to 0)&=V\\
    end(0\to U\to V\to W\to 0)&=(U,W)
\end{align*}

Our objective is to explicitly present the associativity data for this multiplication in the $(\infty,1)$-category $\Corr(\Spaces)$ from \autoref{CorrespondencesSection}. The article \cite{KapranovDyckerhoff} introduces the notion of 2-Segal conditions (we recall their definition in \autoref{2Segal}) and explores various examples of Hall algebra-like constructions in which these conditions imply (higher) associativity.
In the present work we use this notion to construct a functor of  $(\infty,1)$-categories $\AugOrdSet \rightarrow \Corr(\Spaces)$. The image of $\ord{1}$ is the stack of objects of $\CCC$, and thus it is endowed with the structure of algebra in $\Corr(\Stacks)$.

Let us further elaborate on the associativity data in the language of correspondences.
The usual approach to associativity for Hall algebras in the literature can be explained as follows. In the category of correspondences the associativity square is
\[
\stik{1}{
S_1^3 \& \& S_2\times S_1 \ar{ll} \ar{rr} \& \& S_1\times S_1 \\
S_1\times S_2 \ar{u} \ar{d} \& P_2 \ar{dr} \ar{l} \& {} \& P_1 \ar{ul} \ar{r}\&  S_2 \ar{u} \ar{d} \\
S_1\times S_1 \& \& S_2 \ar{ll} \ar{rr} \& \& S_1
}
\]
To show that this square commutes in the 1-category of correspondences one takes pullbacks $P_1$ and $P_2$ corresponding to compositions of the upper and right and left and lower sides and shows that they are isomorphic.
In \cite{KapranovDyckerhoff} it is shown that such isomorphism can be constructed whenever the simplicial stack $S$ associated to the category $\CCC$ is 2-Segal using relations between $S_1,S_2,S_3$.

One way to think about the higher associativity conditions is to consider higher dimensional cubes of correspondences.
We take this approach in \cite{ourCubes}. However to further use this data while working with the categories of sheaves on stacks one needs to formulate the necessary statements in the language of cubical sets model for $\infty$-categories rather than the overwhelmingly more common Kan simplicial set/complete Segal spaces.
Since we use results from the book  \cite{gaitsbook} which is written with the latter model in mind we will use this approach here.

In the next section we will describe a combinatorial construction of a system of correspondences of simplicial sets given by a map
\[N(\AugOrdSet) \rightarrow \BigCorr(\sset^{op})\] 
Applying the extended Waldhausen construction we obtain a map \[N(\AugOrdSet) \rightarrow \BigCorr(\Stacks)\] It turns out that if the simplicial stack $S$ is 2-Segal, the image of this map lands in the monoidal $(\infty,1)$-category $\Corr(\Stacks)$ (\autoref{thmAlgebra}).
This allows us to use the machinery developed in \cite{gaitsbook} to translate this to the monoidal structure on the $(\infty,1)$-category of D-modules on the stack of objects of a category in \autoref{Transfer}. 

\subsection{The combinatorial construction}
\label{Hcomb}

The functor
\[N(\AugOrdSet) \rightarrow \BigCorr(\Stacks)\]
is constructed essentially in a combinatorial way. Namely we construct the map of simplicial sets from the $N(\AugOrdSet)$ - the nerve of the augmented category of ordered sets - to $\BigCorr(\sset^{op})$. We further apply the extension of the Waldhausen construction described in \autoref{S_ext} $S$ object-wise to obtain a map to $\BigCorr(\Spaces)$. 

\begin{Definition}
For $X\in\AugOrdSet$ define the \emph{augmentation} of $X$, $\aug{X}$ to be ${\Hom}_{\AugOrdSet}(X,\{0 \rightarrow 1\})$ considered as an object in $\OrdSet^{op}$.
\end{Definition}

This sends the totally ordered set with $n$ elements in $\AugOrdSet$ (which we denote $\ord{n}$) to the totally ordered set with $n+1$ elements in $\OrdSet^{op}$ (which we denote $\ordop{n}$).

\subsubsection{Construction for objects}
\label{HcombObjects}
Let $X\in\AugOrdSet$. Each element $x\in X$ determines an embedding of $\aug{\{x\}}$ in $\aug{X}$ as follows: Given a map $\{x\} \to \{0\rightarrow 1\}$, we extend it to a map $X \to \{0 \rightarrow 1\}$ by setting it to $0$ on lower than $x$ elements and $1$ on higher than $x$ elements of $X$. Consider the sub-simplicial set of $\aug{X}$ (considered as an object in $\sset^{op}$) generated by these embeddings. We will denote this simplicial set by $H_{comb}(X)$.

\begin{Example}
The first few values of $H_{comb}$ on the objects of $\AugOrdSet$ are as follows
\begin{itemize}
\item $H_{comb}(\ord{0})=\ordop{0}$
\item $H_{comb}(\ord{1})=\ordop{1}$ 
\item $H_{comb}(\ord{2})$ is the horn 
\includegraphics[scale=0.3]{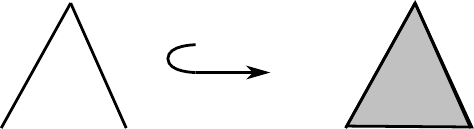}
\item $H_{comb}(\ord{3})$ is \includegraphics[scale=0.3]{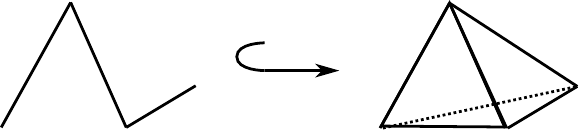}
\end{itemize}
\end{Example}
Applying the Waldhausen construction we get  
\begin{itemize}
\item $\ordop{0}\xmapsto{S}S_0(C)=\point$
\item $\ordop{1}\xmapsto{S}S_1(\CCC)$ 
\item $\includegraphics[scale=0.3]{Figures/2HornComb.pdf} \xmapsto{S}S_1(\CCC)\times S_1(\CCC)\hookrightarrow S_2(\CCC)$
\item $\includegraphics[scale=0.3]{Figures/3HornComb.pdf}\xmapsto{S}S_1(\CCC)\times S_1(\CCC)\times S_1(\CCC) \hookrightarrow S_3(\CCC)$
\end{itemize}

\subsubsection{Construction for arrows}
\label{HcombArrowConstruction}
Let $X \xrightarrow{f} Y$ be a map in $\AugOrdSet$. We want to associate to it a correspondence in $\sset^{op}$.

Let $y\in Y$ and denote $X_y$ the preimage of $y$ under $f$. Similarly to the above, we get an imbedding $\aug{X_y}$ in $\aug{X}$. Denote the sub-simplicial set generated by these imbeddings for all $y$ by $H_{comb}(f)$. Then we have a natural correspondence \[
H_{comb}(X) \rightarrow{} H_{comb}(f) \leftarrow H_{comb}(Y)
\]

\begin{Example}
The multiplication is the image of the map $\ord{2} \to \ord{1}$, and on the level of $H_{comb}$ this map goes to 
\[
\includegraphics[scale=0.5]{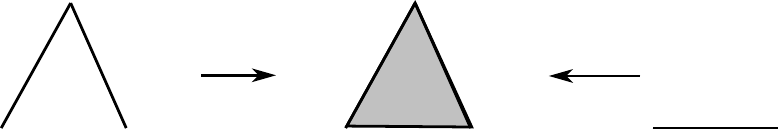}
\]
$S$ then sends the middle object to the short exact sequences, and the maps to restriction to the endpoints or middle respectively. This is the correspondence defining the multiplication in the Hall algebra.
\end{Example}

\subsubsection{Construction for higher morphisms}

Given a 2-cell in $\Nerve{\OrdSet}$
\[
\stik{1}{
X \ar{r}{f} \ar{dr}{\alpha} \& Y \ar{d}{g}\\
{} \& Z
}
\]
Similarly to what we did with arrows, we construct the 2-cell of correspondences

\[
\stik{0.95}{
{} \& {} \& H_{comb}(\alpha) \ar[leftarrow]{dl} \ar[leftarrow]{dr}\& {} \& {} \\
{} \& H_{comb}(f) \ar[leftarrow]{dl} \ar[leftarrow]{dr}\& {} \& H_{comb}(g)\ar[leftarrow]{dl} \ar[leftarrow]{dr}\& {} \\
H_{comb}(X) \& {} \& H_{comb}(Y)\& {} \& H_{comb}(Z)
}
\]
where $H_{comb}(\alpha)$ is the sub simplicial set in $\aug{X}$ generated by the imbeddings of $\aug{\alpha^{-1}(z)}$ for all $z\in Z$.

\begin{Remark}
\label{rem-HcombId}
Note that $H_{comb}(Y)=H_{comb}(\Id_Y)$ so the above can be rewritten as

\begin{equation}
\label{HcombCorrCellEquation}
\cnstik{0.9}{-3.5}{
{} \& {} \& H_{comb}(g\circ\Id_Y\circ f) \ar[leftarrow]{dl} \ar[leftarrow]{dr}\& {} \& {} \\
{} \& H_{comb}(f\circ\Id_X)=H_{comb}(\Id_Y\circ f) \ar[leftarrow]{dl} \ar[leftarrow]{dr}\& {} \& H_{comb}(g\circ\Id_Y)=H_{comb}(\Id_Z\circ g)\ar[leftarrow]{dl} \ar[leftarrow]{dr}\& {} \\
H_{comb}(\Id_X) \& {} \& H_{comb}(\Id_Y)\& {} \& H_{comb}(\Id_Z)
}
\end{equation}

\end{Remark}

From these examples the general definition is now straightforward:

\begin{Definition}
To an $n$-cell in $\Nerve{\AugOrdSet}$, i.e. a sequence of $n$ composable maps $f_1,\ldots,f_n$, we assign the $n$-cell of correspondences which has the apex $H_{comb}(f_n\circ\cdots\circ f_1)$, defined in the same way as above and the general shape of a square in the diagram is
\[
\stik{1}{
H_{comb}(g) \ar{r} \ar{d} \& H_{comb}(g\circ h) \ar{d}\\
H_{comb}(f\circ g) \ar{r} \& H_{comb}(f\circ g\circ h)
}
\]
\end{Definition}

Altogether we obtain a map 
\[
H_{comb}:\Nerve{\AugOrdSet}\rightarrow \BigCorr(\sset^{op})
\]
\subsection{Higher algebra structure for the Hall category and the 2-Segal conditions}
Applying the extended Waldhausen construction  objectwise to the combinatorial construction above we obtain the map
\[
\BigCorr(S)\circ H_{comb}:\Nerve{\AugOrdSet}\rightarrow \BigCorr(\Spaces)
\]

In this section we will show that when the simplicial stack $S$ is 2-Segal this map is in fact a monoidal functor of $(\infty,1)$-categories. In particular it defines canonical algebra structure on $\BigCorr(S)\circ H_{comb}(\ord{1})$ - the stack of objects of $\CCC$. 

Following \cite{gaitsbook} we will consider the $(\infty,1)$-category $\Corr(\Stacks)$ from \autoref{def:corr}. $\Corr(\Spaces)$ is the subspace of  $\BigCorr(\Spaces)$ consisting of the grids with all inner squares being pullbacks.
\begin{Remark}
$\Corr(\Spaces)$ has a structure of a monoidal category given by the Cartesian product in $\Spaces$.
\end{Remark}
It turns out that that the functor $\BigCorr(S)\circ H_{comb}$ lands in $\Corr(\Stacks)$. As a consequence we obtain:
\begin{Theorem}
\label{thmAlgebra}
Let $S$ be a 2-Segal simplicial stack, then $\BigCorr(S)\circ H_{comb}(\ord{1})$ has a canonical structure of algebra in $\Corr(\Stacks)$.
\end{Theorem}

\begin{proof}
First let us show that 2-Segal conditions imply that $\BigCorr(S)\circ H_{comb}$ lands in $\Corr(\Stacks)$.
Note that any cell in $\AugOrdSet$ can be decomposed into a disjoint union of cells ending in $\ord{1}$. Recall that for an ordered set $X$, $H_{comb}(X)$ is a sub-simplicial set of $\aug{X}$ - the simplicial set represented by the ordered set of maps from $X$ to the interval $\{0<1\}$. As a result it has two distinguished points given by the constant maps, which we will call entry and exit points. By construction $H_{comb}$ takes disjoint unions of cells in $\AugOrdSet$ to (correspondences) of simplicial sets glued at corresponding exit/entry points. The gluing in $\sset$ corresponds to taking a product in $\sset^{op}$, and hence $\BigCorr(S)$ takes the above to products of correspondences. (This is a consequence of the fact that the extended Waldhausen construction $S$ is a Kan extension along the Yoneda embedding $\OrdSet^{op} \rightarrow \sset^{op}$, so sends products over a point to products over $S(\point)=\pt$).

Altogether this shows that we can reduce to proving our statement for the cells ending in $\ord{1}$. Same argument also shows that we may assume all maps are onto by splitting off cells starting with $\ord{0}$.

Let us first consider the case of a 2-cell
\[
\stik{1}{
\ord{m}\ar{r}{f} \ar{dr} \& \ord{n}\ar{d}\\
{} \& \ord{1}
}
\]
Its image under $H_{comb}$ includes the square
\[
\stik{1}{
H_{comb}(\ord{n})\ar{r}\ar{d} \& H_{comb}(f) \ar{d}\\
\aug{n}=H_{comb}(\ord{n}\to\ord{1}) \ar{r} \&  \aug{m}=H_{comb}(\ord{m}\to\ord{1})
}
\]

What we want is that $S$ sends this square to a pullback square (note that $S$ is contravariant, i.e. the resulting square of stacks has its arrows inverted). This exactly corresponds to the 2-Segal condition coming from the polygonal decomposition described by inscribing an $n+1$-gon into an $m+1$-gon according to the map $f$ (i.e. placing the interval $[i,i+1]$ astride its preimage under $f$) and so is a pullback since $S$ is a 2-Segal stack.

Similarly, for a general cell, consider again the situation where all maps are onto, then each square appearing in the image can be decomposed (monoidaly) into squares coming from a sequence of maps 
\[
A \xrightarrow{h} B \xrightarrow{g} C \xrightarrow{f} \ord{1}
\]
i.e. squares of the form \[
\stik{1}{
H_{comb}(g) \ar{r} \ar{d} \& H_{comb}(g\circ h) \ar{d}\\
H_{comb}(f\circ g) \ar{r} \& H_{comb}(f\circ g\circ h)
}
\]
(see \autoref{rem-HcombId}).
These kind of squares going to a pullback is exactly the condition of $S$ being adapted to the decomposition described by the nested polygons determined by the sequence of maps $f,g,h$.

As explained above $\BigCorr(S)\circ H_{comb}$ takes disjoint union of cells in $\AugOrdSet$ to products in $\Corr(\Spaces)$, i.e it is monoidal. This endows the image of $\ord{1}$ with the structure of algebra.
\end{proof}

It turns out that the converse of \autoref{thmAlgebra} is true as well. Altogether we have the following:
\begin{Theorem}
\label{2Segalequivalent}
The simplicial stack $S$ is 2-Segal iff the functor $\BigCorr(S)\circ H_{comb}$ lands in the $(\infty,1)$ category $\Corr(\Stacks)$.
\end{Theorem}
An analog of this theorem for the associativity data presented by cubes was proven in \cite{ourCubes}.
\begin{proof}
\leavevmode\newline
We have just shown that if $S$ is 2-Segal $\BigCorr(S)\circ H_{comb}$ lands in $\Corr(\Stacks)$.
For the other direction, note that any polygonal decomposition can be described by a sequence of polygonal nestings, and by the above $S$ must be compatible with any such decomposition.
\end{proof}

\begin{Notation}
We will call the functor  $\BigCorr(S)\circ H_{comb}$ "the geometric Hall algebra" and denote it by $H_{geo}$.
\end{Notation}
\section{Transfer of Structure}

\label{Transfer}
The construction in \autoref{Hcomb} produces a functor $\geoHall:\AugOrdSet\to\CorrStacks$ which gives an algebra object in $\Corr(\Stacks)$. Our goal is to transfer this structure to an algebraic setting $\AAA$. In the most general setting what we need is a monoidal functor $T:\Corr(\Stacks)\to\AAA$.

However in many cases the functor $\geoHall$ actually lands in some smaller subcategory $\Corr(\Stacks)_{\SSS,\PPP}$ of $\Corr(\Stacks)$ where $\SSS, \PPP$ are classes of morphisms in $\CCC$ closed on composition. We note the following important facts:

 \begin{Proposition}
\label{HereditarySm-PrProposition}
\begin{enumerate}
    \item For any exact category $\CCC$, the functor $\geoHall$ lands in the subcategory $\Corr(\Stacks)_{all,proper}$ whose objects are Artin stacks locally of finite type.
    \item Suppose $\CCC=\Rep(\Quiver)$ is the category of representations of a quiver then $\geoHall$ lands in $\Corr(\Stacks)_{smooth,proper}$.
\end{enumerate}
\end{Proposition}

\begin{proof}
By \autoref{WaldhausenStacksSection} composing Waldhausen construction with $H_{comb}$ gives Artin stacks locally of finite type.

Considering the diagram \autoref{HcombCorrCellEquation} we see that all left facing (horizontal in $\Corr$ terminology) maps are of the form \[
H_{comb}(f)\to H_{comb}(g\circ f)
\]
and all right facing (vertical) maps are of the form\[
H_{comb}(f)\to H_{comb}(f\circ g)
\]
Therefore we need to prove:
\begin{enumerate}
    \item The map of stacks induced by a simplicial map of the form $H_{comb}(f)\to H_{comb}(f\circ g)$ is always proper.
    \item The map of stacks induced by a simplicial map of the form $H_{comb}(f)\to H_{comb}(g\circ f)$ is always smooth.
\end{enumerate}

Recall from \autoref{thmAlgebra} that squares of the form \[
\stik{1}{
H_{comb}(g) \ar{r} \ar{d} \& H_{comb}(g\circ h) \ar{d}\\
H_{comb}(f\circ g) \ar{r} \& H_{comb}(f\circ g\circ h)
}
\]
go to pullbacks of stacks. 
i.e., recalling that the Waldhausen construction is contravariant, we have that the square of stacks
\[
\stik{1}{
\geoHall(f\circ g\circ h) \ar{r} \ar{d} \& \geoHall(g\circ h) \ar{d}\\
\geoHall(f\circ g) \ar{r} \& \geoHall(g)
}
\]
is a pullback.
Since pullbacks of proper and smooth maps are proper and smooth maps respectively, we have that if $\geoHall(g\circ h) \to \geoHall(g)$ is proper then $\geoHall(f\circ g\circ h)\to \geoHall(f\circ g)$ is proper.  Similarly $\geoHall(f\circ g) \to \geoHall(g)$ is smooth implies $\geoHall(f\circ g\circ h) \to \geoHall(g\circ h)$ is smooth. 
This means we can reduce to proving

\begin{Proposition}

\begin{enumerate}
    \item The map of stacks induced by a simplicial map of the form \[H_{comb}(\Id)\to H_{comb}(\Id\circ g)\] is always proper.
    \item When $\CCC=\Rep(\Quiver)$ the map of stacks induced by a simplicial map of the form \[H_{comb}(\Id)\to H_{comb}(g\circ \Id)\] is always smooth.
\end{enumerate}
\end{Proposition}

In both cases, we can reduce to the situation where $g:X\to\point$ because the general situation is a product of those. Moreover, any map $g:X\to\point$ can be decomposed into a sequence of maps, each of which is a disjoint union of the maps $\ord{2}\to\ord{1}$ and $\Id_{\ord{1}}$. Since both proper and smooth are closed on composition we are reduced to the case $g:\ord{2}\to\ord{1}$.

For the proper case the resulting map of stacks is the map from the stack of $1$-flags in $\CCC$ to the stack of objects in $\CCC$ defined by $V\hookrightarrow W \mapsto W$. This map is obviously faithful and the fiber is the variety of subobjects of $W$ which is a projective variety.

For the smooth case the resulting map of stacks is the map from the stack of $1$-flags in $\CCC=\Rep(\Quiver)$ to the $2$-fold product of the stack of objects in $\CCC$ defined by \[
V\hookrightarrow W \mapsto (V,W/V)
\]
This map (of global quotients) can be presented by an equivariant smooth map of varieties (see e.g. \cite[\S1.3]{SchiffmannHallCategory}).

\end{proof}

\begin{Remark}
For the case of a general stable $\infty$-category (replacing the stack $S_2\CCC$ with a version coming from the derived category of $\CCC$ as in \cite{ToenDerivedHallAlgebra}) a computation shows that the relative tangent complex has $i^\text{th}$ homology isomorphic to $\Ext^{i-1}(B,A)$. We see then that $\CCC$ having cohomological dimension 1 is equivalent to this tangent complex being supported in non-positive degrees. This is the version of smoothness condition one can further explore in this setting. 
\end{Remark}

\begin{Remark}
As outlined in \autoref{Int:green twist}, the fact that for a hereditary category the functor $\geoHall$ lands in $\Corr(\Stacks)_{smooth,proper}$ provides a more conceptual explanation of the difference between the hereditary and non-hereditary case with respect to the bi-algebra structure on the Hall algebra. We study this in more detail in \cite{ourGeometricHall2}.
\end{Remark}

\begin{Proposition}
\label{propTransfer}
Suppose $T:\Stacks\to\AAA^{op}$ is monoidal and for any map $p\in\PPP$ we have that $T(p)$ has a right adjoint. Suppose also that for any Cartesian square
\[
\stik{1}{
X\ar{r}{f} \ar{d}{h}\& Y \ar{d}{g}\\
Z\ar{r}{k}\& W
}
\]
where $f,k\in\SSS$ and $g,h\in\PPP$ the image under $T$ has an invertible right mate. Then $T$ induces a monoidal functor $\Corr(\Stacks)_{\SSS,\PPP}\rightarrow\AAA$. 
\end{Proposition}

\begin{proof}
Extending to $\Corr$ means we need to evaluate $T$ on grids. We do this by simply choosing right adjoints for all verticals. i.e. consider a grid
\[
\stik{1}{
\bullet  \& \bullet \ar{l}{s} \ar{d}{p} \& \bullet \ar{d}{p} \ar{l}{s} \ar[no head,dotted]{r}\& \bullet \ar{d}{p} \& \ar{l}{s}\bullet \ar{d}{p}\\
{}  \& \bullet \& \bullet  \ar[no head,dotted]{r} \ar[no head,dotted]{dr}\ar{l}{s} \& \bullet \ar[no head,dotted]{d}  \& \bullet \ar{l}{s} \ar[no head,dotted]{d}\\
{} \& {} \& {} \& \bullet \& \bullet \ar{l}{s} \ar{d}{p}\\
{} \& {} \& {} \& {} \& \bullet
}
\]
By sending the $s$'s to $T(s)$'s and the $p$'s to the right adjoints of $T(p)$'s and forming the mates of the squares, we get a simplicial cell in $\AAA$.
\end{proof}

\begin{Remark}
This is an extension of the notion \emph{Theory with Transfer} introduced in \cite[Definition 8.12]{KapranovDyckerhoff} for monoidal $\infty$-categories.
\end{Remark}

We finish with the following straightforward corollary to \autoref{thmAlgebra}:

\begin{Corollary}
\label{corrAlgebra}
Suppose the functor $\geoHall$ lands in the subcategory $\Corr(\Stacks)_{\SSS,\PPP}$ of $\Corr(\Stacks)$.  Composing $\geoHall$ with a monoidal functor 
\[T:\Corr(\Stacks)_{\SSS,\PPP}\rightarrow \AAA\]
endows $T\circ\geoHall(\ord{1})$ with the structure of an algebra object in $\AAA$.
\end{Corollary}
\begin{Definition}
\label{def:AlgebraObject}
For a theory with transfer $T$ landing in $\AAA$ we define $\Alg{T}$ to be the $\AAA$-algebra object defined by $T\circ\geoHall$.
\end{Definition}

\subsection{Transfer to stable \texorpdfstring{$\infty$}{infinity}-categories}
\label{CatTransfer}

The following example of theory with transfer is of central interest to us in this article.

Consider the monoidal functor $\Dmod^*:\Stacks \rightarrow \Cat^{op}$ associating to a stack $X$ the ($\infty$,1)-category of bounded coherent complexes of D-modules. This functor is an extension of the functor of D-modules for smooth varieties. For the construction of the functor and discussion of its properties see \cite{crystals,gaitsbook}. In particular it satisfies the requirements of \autoref{propTransfer} for $\SSS,\PPP=all,proper$. We denote the corresponding extension to $\Corr(\Stacks)_{all,proper}$ by $\CatTransfer$. 

\begin{Remark}
\label{remDmodQuiver}
In the case of a quotient stack $X//G$ the functor $\Dmod^*$ assigns the stable $(\infty)$-category whose homotopy category is the $G$-equivariant category of coherent D-modules from \cite{BBGDmod}. Note that the moduli stack of objects of $\Rep_{\CC}\Quiver$ is a disjoint union of stacks of this form.
\end{Remark}

\begin{Theorem}
\label{thmCatTransferAlg}
$\CatTransfer\circ \geoHall$ produces an algebra object in $\Cat$.
\end{Theorem}

\begin{proof}
By \autoref{thmAlgebra} and \autoref{HereditarySm-PrProposition}, $\geoHall$ produces an algebra object in $\Corr(\Stacks)_{all,proper}$. Composing with the monoidal functor $\CatTransfer$ yields an algebra object in $\Cat$.
\end{proof}

Recall that the multiplication of the algebra given by $\geoHall$ is defined by the correspondence\[
\stik{1}{
{} \& S_2\CCC \ar{dl}[above, xshift=-0.5em]{end} \ar{dr}{mid} \& {} \\
S_1\CCC\times S_1\CCC \& {} \& S_1\CCC
}
\]
and so the multiplication in the algebra defined by $\CatTransfer\circ \geoHall$ is given by $m:=mid_*\circ end^*$.
\begin{Corollary}
\label{corCatTransferMonoidal}
The functor \[m:\Dmod(\Ob_{\CCC})\otimes\Dmod(\Ob_{\CCC})\rightarrow\Dmod(\Ob_{\CCC})\] equips $\Dmod(\Ob_{\CCC})$ with a structure of monoidal $(\infty,1)$-category.
\end{Corollary}

\begin{Remark}
\label{remLeftRightAdjoint}
Note that since $mid$ is always proper, this is the same as $mid_!\circ end^*$ which often appears in the literature. Additionally, when $\CCC=\Rep(\Quiver)$ the map $end$ is smooth and so $end^*$ differs from $end^!$ by a shift. As a result this functor (and similarly all the other functors appearing in this construction) has both a left and right adjoint. This point will be important when studying the bialgebra structure of the Hall category in \cite{ourGeometricHall2}.
\end{Remark}

\subsection{Transfer to \texorpdfstring{$\VectNu$}{Vect-nu}}
\label{VectTransfer}

Let us take Grothendieck groups of the $(\infty,1)$-category construction above. The shift functors endow $\Kgroup(\Dmod(X))$ with the structure of an algebra over $\ZZ[\nu,\nu^{-1}]$ by defining
\[\nu[M]=[M[1]], \nu^{-1}[M]=M[-1]\]

Thus we obtain a functor \[
\VectNuTransfer:\Corr(\Stacks)\to\VectNu
\]

\begin{Proposition}
$\VectNuTransfer(\Ob_{\CCC})$ has the structure of associative algebra over $\ZZ[\nu,\nu^{-1}]$.
\end{Proposition}
\begin{proof}
This is a direct consequence of \autoref{thmCatTransferAlg} and the fact that the functors $\square^*$ and $\square_*$ commute with  shifts.
\end{proof}

\section{Categorification of quantum groups.}
\label{secKLR}
\subsection{The stable \texorpdfstring{$\infty$}{infinity} Hall category}
\label{secHallCat}
In this section we will fix $\CCC=\Rep_{\CC}\Quiver$ for $\Quiver=(I,H)$ - a quiver with $I$ the set of vertices and $H$ the set of edges. We have source and target maps $s,t:H\rightarrow I$. We will assume $\Quiver$ has no loop edges, i.e. $s(h)\neq t(h)$ for any $h\in H$. 

To a quiver $\Quiver$ we can assign a Cartan datum and so a Kac-Moody algebra $\mathfrak{g}$. We denote by  $U_\nu(\mathfrak{g}_{\Quiver})$ the quantized enveloping algebra, and by $\dot{U}_\nu(\mathfrak{g}_{\Quiver})$ its idempotented integral form introduced by Lusztig. $\dot{U}_\nu(\mathfrak{g}_{\Quiver})$ is a $\ZZ[\nu,\nu^{-1}]$-algebra.

Applying our geometric Hall algebra construction followed by transfer from \autoref{CatTransfer} to $\CCC$ produces a monoidal $(\infty,1)$-category $\Alg{\CatTransfer}$ whose Grothendieck group contains, as a subalgebra, the negative half of $\dot{U}_\nu(\mathfrak{g}_{\Quiver})$.

When $\Quiver$ is of finite type we have $K(\Alg{\CatTransfer})\cong \Zquantum$. In general, one can define a certain subcategory $\HallCat$ of $\Alg{\CatTransfer}$ whose Grothendieck group is isomorphic to $\Zquantum$. We will define $\HallCat$ and explain the relationship to the quantum group in this section. 

Our definition of $\HallCat$ is an alteration of the definition, essentially due to Lusztig, of an additive monoidal category $\QQQ$ (the monoidal structure on $\QQQ$ is constructed for example in \cite{SchiffmannHallCategory}). The product on both $\QQQ$ and $\HallCat$ is given by the correspondence of stacks 
\begin{equation}
\label{CorrMult}
\stik{1}{
{} \& S_2\CCC \ar{dl}[above, xshift=-0.5em]{end} \ar{dr}{mid} \& {} \\
S_1\CCC\times S_1\CCC \& {} \& S_1\CCC
}
\end{equation}
Both categories are generated under this product by certain elementary sheaves on the components of the stack $S_1\CCC=\Ob_{\CCC}$. The difference is that $\QQQ$ is defined to be an additive category and $\HallCat$ is, of course, a stable $\infty$-category. The higher coherences for the monoidal structure on $\HallCat$ are explicitly given by the extended Waldhausen construction of \autoref{WaldhausenStacksSection}.The bigger category $\HallCat$ is a correct environment for constructing the categorification of the bialgebra conditions in \cite{ourGeometricHall2}.

To elaborate on the construction of the above product: 
we compose the functor 
\[H_{geo}:N(\AugOrdSet) \rightarrow \Corr(\Stacks)\] with the functor \[\CatTransfer:\Corr(\Spaces)\rightarrow \Cat\] from \autoref{CatTransfer} to get a monoidal functor 
\[N(\AugOrdSet)\rightarrow \Cat\] 
By \autoref{corrAlgebra} this is an algebra object which we denoted $\Alg{\CatTransfer}$ in the category of $(\infty,1)$-categories. In particular, $\Alg{\CatTransfer}$ is a monoidal stable $(\infty,1)$-category. We will denote the product by $\star$.

Note that by \autoref{remDmodQuiver}, $\CatTransfer\circ\geoHall(\ord{1})=\Dmod_G(\Ob_{\CCC})$. Let us now define a monoidal subcategory $\HallCat$ of $\Dmod_G(\Ob_{\CCC})$. 

Let us briefly recall certain facts about $\CCC:=\Rep_{\CC}\Quiver$. The objects of $\CCC$ are pairs $(V=\bigoplus_{i\in I}V_i,(x_h)_{h\in H})$ of finite dimensional $I$-graded vector spaces and collections of morphisms $x_h:V_{s(h)}\rightarrow V_{t(h)}$. 
To every object of $\CCC$ we can associate its dimension vector  $\alpha \in \NN^{I}$. The stack of objects of $\CCC$ of dimension $\alpha$ can be constructed as a quotient stack $\Ob_{\CCC}^{\alpha}=E//G$, where
\[E=\bigoplus_{h\in H}\Hom(V_{s(h)},V_{t(h)}), G=\Pi_{i\in I}GL(V_{\alpha_{i}})\] 
$g\in G$ acts by $g\cdot(y_{h})_{h\in H}=(g_{t(h)}y_hg^{-1}_{s(h)})_{h\in H}$. The isomorphism classes of objects of $\CCC$ of dimension $\alpha$ correspond to the points of the stack $\Ob_{\CCC}^{\alpha}\cong E//G$. The stack of objects of the category $\CCC$ decomposes as $\Ob_{\CCC}=\bigsqcup_{\alpha \in \NN^I}\Ob_{\CCC}^{\alpha}$.The isomorphism classes of simple objects correspond to the standard basis ${\{\epsilon_i\}}_{i \in I}$ of $\NN^I$. Since our quiver has no loops there is only one such class corresponding to each such dimension vector. 

\begin{Definition}
For an $n$-tuple $\underline{\alpha}=(\alpha_1,\ldots,\alpha_n)$ of dimension vectors, define the \emph{Lusztig sheaf} $\Lusztig_{\underline{\alpha}}$ to be the product $\OOO_{\alpha_1}\star\cdots\star\OOO_{\alpha_n}$.
\end{Definition}

\begin{Definition}
\label{defH}
$\HallCat$ is the smallest stable subcategory of $\Alg{\CatTransfer}$ containing the D-modules $\Lusztig_{\underline{\alpha}}$ where all  $\{\alpha_i\}_{1\leq i\leq n}$ are the dimension vectors corresponding to simple objects of $\CCC$.
\end{Definition}

By taking the Grothedieck group $K(\HallCat)$ we obtain a subalgebra of $\Alg{\VectNuTransfer}$ over $\ZZ[\nu,\nu^{-1}]$ 
where $\nu,\nu^{-1}$ act as 
\[\nu[M]=[M][1],\nu^{-1}[M]=M[-1]\].

\begin{Proposition}
\label{propQuantumGroup}
The algebra $K(\HallCat)$ is isomorphic to Lusztig's integral form of the quantum group $\Zquantum$.
\end{Proposition}

\autoref{propQuantumGroup} follows from the theorem by Lusztig that constructs an isomorphism between $\Zquantum$ and the Grothendieck group of the Lusztig's category $\QQQ$. $\QQQ$ is defined to be an additive subcategory of the derived category of complexes of equivariant constructible sheaves on $\Ob_{\CCC}$. Let us briefly recall the definition of $\QQQ$.

For every stack $\Ob_{\CCC}^{\alpha}=E//G$ from the decomposition  $\Ob_{\CCC}=\bigsqcup_{\alpha \in \NN^I}\Ob_{\CCC}^{\alpha}$ consider the derived category of bounded complexes of $G$-equivariant constructible sheaves $\DShcomplex{(E)}$ (see \cite{BernsteinLunts}). The correspondence of stacks \autoref{CorrMult} defines monoidal structure on $\DShcomplex{(E)}$ with the product given by the functor $m:=mid_!\circ end^*$ (see e.g. \cite{SchiffmannHallCategory} for the proof of this statement). Let us denote this product also by $\star$.

Let $\mathbb{1}_{\alpha}$ denote the perverse extension of the constant sheaf on the component $\Ob_{\alpha}$. 
\begin{Definition}
\label{defQ}
$\QQQ$ is defined to be the smallest triangulated subcategory of $\DShcomplex{(\Ob_{\CCC})}$ closed under $\star$, shifts and taking direct summands and containing all complexes $\mathbb{1}_{\alpha}$ for the dimension vectors $\alpha$ corresponding to the simple objects of $\CCC$. 
\end{Definition}
We have the following:

\begin{unnu-theorem}[Lusztig]
\label{thmLusztigQuantum}
$K(\QQQ)\cong \Zquantum$
\end{unnu-theorem}

\begin{proof}[Proof of \autoref{propQuantumGroup}]
The de-Rham functor induces a natural map 
\[K(\Dmod(\Ob_{\CCC}))\rightarrow K(\DShcomplex{\Ob{\CCC}})\]
whose restriction to the Grothendieck group of the category of regular holonomic D-modules is an isomorphisms.Since $\mathcal{O}_{\Ob_{\CCC}^{\alpha}}$ and their products are regular holonomic we obtain an isomorphism
\[K(\HallCat)\cong K(\QQQ)\]
This isomorphism is natural hence it induces an isomorphism of algebras.
By the theorem above we get 
\[
K(\HallCat)\cong \Zquantum
\]

\end{proof}

\begin{Remark}
\label{remFiniteTypeQuiver}
For the case of $\Quiver$ of finite type the number of $G$ orbits is finite for each $\Ob_{\CCC}^\alpha=E//G$. 
Hence we have $\QQQ=\DShcomplex{\Ob_{\CCC}}$ (see \cite[Theorem 2.8]{SchiffmannHallCategory}). 
Also since in this case all the $G$-equivariant D-modules are regular holonomic (see \cite{KashiwaraRegHol}) the category $\Dmod(\Ob_{\CCC})$ and the category $\DShcomplex{(\Ob_{\CCC})}$ are equivalent. 
Therefore in this case $\HallCat=\Dmod(\Ob_{\CCC})=\Alg{\CatTransfer}$ and $K(\Alg{\CatTransfer})\cong\Zquantum$
\end{Remark}

\subsection{Relation to the KLR categorification}
\label{subKLR}
In this section we would like to explain the connection between the monoidal category $\HallCat$ from \autoref{defH} and the Khovanov-Lauda-Rouquier (KLR) categorification \cite{KL1,KL2,R} of the quantum group $U_{\sqrt{q}}(\mathfrak{g}_{\Quiver})$ associated to the quiver $\Quiver$.

In \cite{KL1,KL2} the authors construct a family of graded algebras $R(\nu)$ indexed by dimension vectors $\nu \in \NN^I$. The categorification of the quantum group $U_{\sqrt{q}}(\mathfrak{g}_{\Quiver})$ is given by the categories $R(\nu)-\Mod^{gr,proj}$ of $\ZZ$- graded projective finitely generated modules over $R(\nu)$. The authors prove that 
\[
\bigoplus_{\nu \in \NN^I}
K(R(\nu)-\Mod^{gr,proj})  \cong \dot{U}_{\sqrt{q}}(\mathfrak{g}_{\Quiver})
\] 
as $\ZZ[q,q^{-1}]$ algebras where the powers of $q$ act by shifts of the grading and the product on the left hand side is given by induction functors between the grades.

Let us explain how this category is related to $\HallCat$.

\begin{Definition}
Define the \emph{Lusztig algebroid} $\LusztigAlgebroid$ to be the $\Ainf$-category with
\begin{itemize}
    \item Objects: $\underline{\alpha}=(\alpha_1,\ldots,\alpha_n)$ with $\alpha_i$ being dimension vectors.
    \item Morphisms: $\Maps(\underline{\alpha},\underline{\beta})=\Maps(\Lusztig_{\underline{\alpha}},\Lusztig_{\underline{\beta}})$.
\end{itemize}
\end{Definition}

\begin{Proposition}
 $\HallCat$ is equivalent to the category $\Perf{\LusztigAlgebroid}$ of perfect modules over $\LusztigAlgebroid$.
\end{Proposition}

\begin{proof}
$\HallCat$ is graded by dimension vectors, and for each dimension vector $\nu$ we have that $X_{\nu} := \bigoplus_{|\underline{\alpha}|=\nu} \Lusztig_{\underline{\alpha}}$ is a compact generator. Hence we obtain the functor $\Maps(X_{\nu},-)$ from $\HallCat$ to the category of modules over the $\Ainf$-algebra $\AinfEnd(X_{\nu})$. As we explain below this algebra is formal and the images of Lusztig sheaves $\Lusztig_{\underline{\alpha}}$ are quasi-isomorphic to graded projective modules over its cohomology. By definition of $\HallCat$ the image of $\Maps(X_{\nu},-)$ is the the smallest stable subcategory containing these images, hence it is equal to $\Perf(\AinfEnd(X_{\nu}))$. 

The proposition now follows.
\end{proof}

In this language, the monoidal structure is expressed via tensoring with bimodules. Namely for each pair of dimension vectors $\nu,\nu'$ we have the bimodule $M_{\nu,\nu'}:=\Maps(X_{\nu+\nu'},X_\nu\otimes X_{\nu'})$.

The Lusztig sheaves $\Lusztig_{\underline{\alpha}}$ are pure Hodge sheaves of weight $0$, since they are given by pushforwards of pure Hodge sheaves along proper maps. It follows that algebras $\AinfEnd{X_{\nu}}$ are formal, i.e. quasi-isomorphic to $\Ext^*(X_\nu,X_\nu)$. Moreover the bimodules $M_{\nu,\nu'}$ are formal.

Varagnolo and Vasserot in \cite{VaragnoloVasserot} prove the following  

\begin{unnu-theorem}[Varagnolo-Vasserot, \cite{VaragnoloVasserot}]
Let $R(\nu)$ be the KLR algebra corresponding to the dimension vector $\nu$, then as a graded algebra,
$\Ext^*(X_\nu,X_\nu)$ is Morita equivalent to $R(\nu)$.
\end{unnu-theorem}

\begin{unnu-theorem}[Varagnolo-Vasserot, \cite{VaragnoloVasserot}]
The graded piece of the Lustztig category $\QQQ$ corresponding to the dimension vector  $\nu$ is equivalent to the category of graded projective finite dimensional modules over the graded algebra $\Ext^*(X_{\nu},X_\nu)$. 
\end{unnu-theorem}

In particular, \cite{VaragnoloVasserot} show that  $\Ext^*(X_\nu,\Lusztig_{\underline{\alpha}})$ is a graded projective module over $\Ext^*(X_{\nu},X_\nu)$.

We can therefore express the monoidal structure in terms of tensoring with the bimodules  $\Ext^*(X_{\nu+\nu'},X_\nu\otimes X_{\nu'})$ (quasi-isomorphic to $M_{\nu,\nu'}$) over $R(\nu)$. Note that in \cite{KL1} the monoidal structure is given - in contrast with the above - by induction functors between the $R(\nu)$. i.e. the multiplication functor \[
\HallCat_\nu\otimes\HallCat_{\nu'}\cong R(\nu)\dashmod\otimes R(\nu)\dashmod\to R(\nu+\nu')\dashmod\cong\HallCat_{\nu+\nu'}
\]
is given by induction along an inclusion $R(\nu)\otimes R(\nu')\subset R(\nu+\nu')$. 

\begin{Proposition}
The monoidal structure on $\HallCat$ given by induction and the one given by tensoring with the $M_{\nu,\nu'}$ coincide.
\end{Proposition}

\begin{proof}
The bimodule $\Ext^*(X_{\nu+\nu'},X_\nu\otimes X_{\nu'})$ corresponds to a summand of $R(\nu+\nu')$, whose complement is annihilated by the left action of  $R(\nu)\otimes R(\nu')$, and so tensoring with this bimodule is isomorphic to induction.
\end{proof}

\begin{Remark}
In \cite[\S 4.9]{VaragnoloVasserot} Varagnolo and Vasserot also note that the multiplication is given by induction, and prove it by comparing K-classes on generators.
\end{Remark}

\begin{Remark}
\label{remFormality}
Note that the KLR categorification is given in terms of graded projective modules over the algebras $Ext^*(X_{\nu},X_{\nu})$. An important point is that while the bimodules defining the monoidal structure are all formal, and so can be considered as graded $R(\nu)$ modules, the same is not true for the bimodules defining comultiplication (see \cite[\S 3.7]{SchiffmannHallCategory}). This is another reason to think that in order to categorify the bialgebra structure properly, it is essential to work with modules over $\AinfEnd(X_{\nu})$ instead.
\end{Remark}

\subsection{Connection to the classical Hall algebra}
\label{secHallConnection}
We would like to briefly summarize the connection between the category $\HallCat$ from \autoref{defH} and the Ringel-Hall algebra. 

Recall the notion of the classical Hall algebra associated to a finitary abelian category $\CCC$: 

\begin{Definition}
\label{defClassicHall}
Let $\CCC$ be a finitary abelian category i.e. $\CCC$ has finitely many simple isomorphism classes and the groups $\Hom$ and $\Ext^{1}$ are finite. We define the \emph{Hall algebra} of $\CCC$, $\HallAlg$, to be the algebra given by
\begin{itemize}
    \item As a space $\HallAlg$ has a basis the isomorphism classes of $\CCC$.
    \item Multiplication is given by the formula
    \[
    [X]\cdot[Y]=\sum_{[Z]} \#\{
    X'\hookrightarrow Z|X'\cong X, Z/X\cong Y\}[Z]
    \]
\end{itemize}
\end{Definition}
$\CCC=\Rep_{\FF_q}(\Quiver)$ satisfies the the requirements above and therefore we can consider $\HallAlg$ associated to it.

\begin{Claim}
$\HallAlg$ is isomorphic to $K(\HallCat)$, the Grothendieck $K$-group of the category $\HallCat$.
\end{Claim}

This statement becomes more transparent if we reformulate the definition of $\HallAlg$ in the language of theory with transfer. Such a reformulation appears in  \cite{KapranovDyckerhoff}. Let us rephrase it using our construction: 

For a stack $X$ defined over $\field$ we define $\StackFunc(X)$ to be the space of finitely supported $\CC$-valued functions on the set of isomorphism classes of $X(\field)$.
We now extend this to a monoidal functor from $\Corr(\Stacks)_{smooth,proper}$ to $\Vect$ which we denote $\DKTransfer$. 

 By \autoref{HereditarySm-PrProposition} the functor $\geoHall$ associated to the category $\Rep(\Quiver)$ lands in $\Corr(\Stacks)_{smooth,proper}$. To a correspondence \[
\stik{1}{
{} \& Z \ar{dl}[above, xshift=-0.5em]{q} \ar{dr}{p} \& {} \\
X \& {} \& Y
}
\]
in $\Corr(\Stacks)_{smooth,proper}$ we first assign the correspondence of groupoids
\[
\stik{1}{
{} \& Z(\FF_q) \ar{dl}[above, xshift=-0.5em]{q(\FF_q)} \ar{dr}{p(\FF_q)} \& {} \\
X(\FF_q) \& {} \& Y(\FF_q)
}
\]
and then we assign the map $\StackFunc(X)\to\StackFunc(Y)$ given by $p_!\circ q^*$ where $q^*$ is the pullback of functions and $p_!$ is "integration along fibres" which is given by the sum $p_!(f)([A])=\sum_{[B]}|\Aut_{X_A}(B)|^{-1}\cdot f(B)$ taken over all isomorphism classes in the groupoid fiber $X_A$.

\begin{Remark}
For the above construction to make sense, all quantities involved need to be finite, and all the resulting functions should be finitely supported. This is ensured by the requirement that the original maps are smooth or proper and that the field is finite.
\end{Remark}

$\DKTransfer$ defined in this way provides a \emph{theory with transfer} in the sense of \cite{KapranovDyckerhoff}.

We then have

\begin{unnu-theorem}[Dyckerhoff-Kapranov]
The algebra $\Alg{\DKTransfer}$ determined by $\DKTransfer\circ \geoHall$ is canonically isomorphic to $\HallAlg$.
\end{unnu-theorem}

Using this reformulation we can expect the relation between $K(\HallCat)$ and $\HallAlg$ to be given by considering the stacks that appear in our construction $\geoHall$ over the field $\kk$ with positive characteristic $q$ rather than over $\CC$ and the $G$-equivariant derived categories $\DShladic{X}$ of $l$-adic sheaves on these stacks. The advantage of working in this setting is that we can use the sheaf-to-function correspondence to relate between $K(\HallCat)$ and $\HallAlg$. 

The same construction as in \autoref{defQ} produces Lusztig's category $\QQQ$ in this setting and we have the isomorphism of $\ZZ[\nu,\nu^{-1}]$-algebras 
\[K(\HallCat)\cong K(\QQQ)\]

It was proven by Lusztig in \cite{LusztigQuivers} that any $\FFF\in\DShladic{X}$ is equipped with a canonical Frobenius equivariant (Weil) structure, i.e. an isomorphism $\FFF\cong \Frob^*(\FFF)$. We can exploit the fact that isomorphism classes of $X(\field)$ are fixed by $\Frob$ to define a function $\Tr(\FFF)\in\StackFunc(X)$ by \[
\Tr(\FFF)(x)=\sum_i (-1)^i\Tr(\Frob|_{H^i(\FFF|_x)})
\]

This map has several useful properties:
 \begin{itemize}
     \item Let $p:X\to Y$ then by the Grothendieck trace formula $\Tr(p_!\FFF)=p_!\Tr(\FFF)$.
     \item $\Tr(\FFF)$ only depends on the Grothendieck class $[\FFF]$ of $\FFF$ and so $\Tr$ defines a map $K(\DShladic{X})\to\StackFunc(X)$.
     \item $\Tr(\FFF[1])=\sqrt{q}\cdot\Tr(\FFF)$
 \end{itemize}
 
 We can sum up as follows:
 
 \begin{Claim}
 $\Tr$ is a natural transformation between the two functors \[X\mapsto K(\DShladic{X})\otimes_{\ZZ[\nu,\nu^{-1}]}\CC=:\KWeilTransfer \text{\phantom{M} ($\nu$ acts on $\CC$ by $\sqrt{q}$)}\] 
 and 
 \[X\mapsto\StackFunc(X)=\DKTransfer(X)\]
 \end{Claim}

This morphism induces a morphism of algebras which, for $X=\Ob_\CCC$ descends to an isomorphism between $K(\HallCat) (\cong K(\QQQ))$ and a \emph{spherical subalgebra} of the Hall algebra of $\Rep_{\FF_q}\Quiver$ (see \cite[Theorem 3.24]{SchiffmannHallCategory}).

\printbibliography
\end{document}